\documentclass [12pt,a4paper,reqno]{amsart} \textwidth 165mm
\makeatletter
\@namedef{subjclassname@2020}{\textup{2020} Mathematics Subject Classification}
\makeatother
\usepackage{amssymb}
\newcommand{\ef}{\end{equation}}
\chardef\bslash=`\\ 

\newtheorem{thm}{Theorem}[subsection]
\renewcommand{\thethm}{%
  \ifnum\value{subsection}>0
    \thesubsection
  \else
    \thesection
  \fi
  .\arabic{thm}%
}
\newtheorem{thm*}{Theorem}
\newtheorem{cor}[thm]{Corollary}
\newtheorem{lem}[thm]{Lemma}

 \newtheorem{defn}[thm]{Definition}

\newtheorem{theorem}{Theorem}[section]
\newtheorem{corollary}[theorem]{Corollary}
\newtheorem{lemma}[theorem]{Lemma}

 \newtheorem{definition}[theorem]{Definition}

 \theoremstyle{remark}
 \newtheorem{remark}[theorem]{Remark}

\newtheorem{acknowledgment*}[thm] {Acknowledgment}

\newcommand{\thmref}[1]{Theorem~\ref{#1}}

\newcommand{\lemref}[1]{Lemma~\ref{#1}}
\newcommand{\corref}[1]{Corollary~\ref{#1}}

 \renewcommand{\sectionmark}[1]{}

 \newcommand{\cupl}{\operatornamewithlimits{\bigcup}\limits}

\newcommand{\doe}{\overset{\text{def}}{=}}
 
\newcommand{\loc} {\operatorname{loc}}

 \date{}
\begin{document}

\title[Correct Solvability of the  general Sturm-Liouville equation]{Principal fundamental system of solutions,\\ the Hartman-Wintner problem  and correct solvability\\ of the general Sturm-Liouville equation}
\author[N. Chernyavskaya]{N. Chernyavskaya}\address{
Department of Mathematics, Ben-Gurion University of the
Negev, P.O.B. 653, Beer-Sheva, 84105, Israel}
\email{nina@math.bgu.ac.il}
\author[L. Shuster]{L. Shuster}
 \address{Department of Mathematics,
 Bar-Ilan University, 52900 Ramat Gan, Israel}
 \email{miriam@math.biu.ac.il}
 \thanks{corresponding author: L. Shuster}

\begin{abstract}
 We study the problem of correct solvability in the space $L_p(\mathbb R),$ $p\in[1,\infty)$ of the equation
$$
-(r(x) y'(x))'+q(x)y(x)=f(x),\quad x\in \mathbb R
$$
under the conditions
$$r>0,\quad q\ge 0,\quad \frac{1}{r}\in L_1(\mathbb R),\quad q\in L_1(\mathbb R).$$
        \end{abstract}
\maketitle

    \baselineskip 20pt

\section{Introduction}\label{introduction}
\setcounter{equation}{0} \numberwithin{equation}{section}

In the present paper, we  continue the study of the equation
\begin{equation}\label{1.1}
 -(r(x) y'(x))'+q(x)y(x)=f(x),\quad x\in \mathbb R,\quad f\in L_p(\mathbb R),\quad p\in[1,\infty)
\end{equation}
developed in \cite{3,4,6,7,8,9}.
In these papers, the problem of the existence of a unique bounded solution of \eqref{1.1} in the space $L_p(\mathbb R),$ $p\in[1,\infty)$ was studied under the following two conditions:
\begin{equation}\label{1.2}
r>0,\quad q\ge0,\quad \frac{1}{r}\in L_1^{\loc}(\mathbb R),\quad q\in L_1^{\loc}(\mathbb R),
\end{equation}
\begin{equation}\label{1.3}
\lim_{|d|\to\infty}\int_{x-d}^x\frac{dt}{r(t)}\cdot\int_{x-d}^x q(t)dt=\infty.
\end{equation}

Below, in contrast with the cited papers, we study the same problem under the condition
\begin{equation}\label{1.4}
r>0,\quad q\ge0,\quad \frac{1}{r}\in L_1(\mathbb R),\quad q\in L_1(\mathbb R).
\end{equation}

To describe our work in a more conceptual way, we make some conventions and state some definitions.
In the sequel, condition \eqref{1.2} is our standing assumption and is therefore never mentioned; the letters $c,$ $c(\cdot)$ are used to denote absolute positive constants which are not essential for exposition and may differ even within a single chain of computations. By a solution of \eqref{1.1} we mean a function $y(\cdot),$ absolutely continuous together with the function $r(\cdot)y'(\cdot)$ and satisfying equality \eqref{1.1} almost everywhere in $\mathbb R.$

We now introduce the following basic notion.

\begin{defn}\label{defn1.1} \cite{4,6}
We say that equation \eqref{1.1} is correctly solvable in a given space $L_p(\mathbb R),$ $p\in[1,\infty)$ if the following assertions hold:
\begin{enumerate}
\item[I)] for every   $f\in L_{p}(\mathbb R)$, there exists a unique solution $y\in L_{p}(\mathbb R)$ of \eqref{1.1};
\item[II)] there exists a constant $c(p)\in(0,\infty)$ such that regardless of the choice of $f\in L_{p}(\mathbb R),$ the solution $y\in L_{p}(\mathbb R) $ of \eqref{1.1} satisfies the inequality
    \begin{equation}\label{1.5}
    \|y\|_{L_p(\mathbb R)}\le c(p)\|f\|_{L_p(\mathbb R)}.
    \end{equation}
    \end{enumerate}

    If I) or II) fails to hold, we say that equation \eqref{1.1} in the space $L_p(\mathbb R)$ is not correctly solvable.
    \end{defn}

    We make the following conventions: for brevity, instead of repeating Definition \ref{defn1.1} we say ``the problem concerning I)--II)'' or ``the question concerning I)--II)''. We often omit the word equation before the symbol \eqref{1.1} and the word ``space'' before $L_p(\mathbb R).$ Instead of $L_p(\mathbb R)$ and $\|\cdot\|_{L_p(\mathbb R)}$, we write $L_p$ and $\|\cdot\|_p.$ Finally, the symbol $p'$ denotes the conjugate number of $p\in(1,\infty).$

    Recall that the main goal of our work is to study the question of I)--II) in the case \eqref{1.4}, and we explain now the reason for focusing on this condition. The question concerning I)--II) was posed and answered in \cite{4} under the condition $r\equiv 1,$ and the problem concerning I)--II)  was solved in \cite{6} under the conditions $r\not\equiv1$ and \eqref{1.3}. It turned out that there is an essential difference between the cases $r\equiv1$ and $r\not\equiv 1.$ Whereas in \cite{4} for $r\equiv1$ the criterion for I)--II) to hold is expressed in terms of a simple requirement for the function $q(\cdot)$ itself, in \cite{6} for $r\not\equiv1$ the criterion for I)--II) to hold is expressed in terms of auxiliary functions (implicit function in $r(\cdot)$ and $q(\cdot)$) which are only defined under condition \eqref{1.3}. Thus, the question regarding I)--II) in the case \eqref{1.4} remained open.

    In the present paper, we propose a new version of the method to study the problem concerning I)--II) used in \cite{6}. Its obvious advantage compared to the approach taken in \cite{6} can be stated as follows. In the case \eqref{1.3}, combining the old and new approaches allows one to simplify, in an essential way, the solution of the problem concerning I)--II), whereas in the case of \eqref{1.4} the new approach to I)--II) guarantees a sufficiently simple solution of this problem by standard analytic techniques (see \S3 and \S5 below).

    The structure of the paper is as follows: In \S2, we collect the preliminaries used in the proofs; \S3 contains the statements of our results and some additional comments; all the proofs are given in \S4; in \S5 we give an example  and further discussion.

    \section{Preliminaries}

    Below we present the definitions and facts that are used in the proofs.

    \subsection{Principal fundamental system of solutions (PFSS)}
    \setcounter{equation}{0} \numberwithin{equation}{section}

    \begin{defn}\label{defn2.1.1} \cite{3} We say that a fundamental system of solutions (FSS) $\{u(x),v(x)\},$ $x\in\mathbb R$ of the equation
    \begin{equation}\label{2.1}
    (r(x)z'(x))=q(x)z(x),\quad x\in\mathbb R
    \end{equation}
    is a principal fundamental system of solutions (PFSS) if the functions $u(x)$ and $v(x)$ possess the following properties:
    \begin{equation}\label{2.2}
    u(x)>0,\quad v(x)>0,\quad u'(x)\le 0,\quad v'(x)\ge0\quad\text{for}\quad x\in\mathbb R,
    \end{equation}
    \begin{equation}\label{2.3}
   r(x)[v'(x)u(x)-u'(x)v(x)]=1\quad\text{for}\quad x\in \mathbb R,
    \end{equation}
    \begin{equation}\label{2.4}
    \lim_{x\to \infty}\frac{u(x)}{v(x)}=\lim_{x\to-\infty}\frac{v(x)}{u(x)}=0.
    \end{equation}
    \end{defn}

    In connection with this definition, we refer the reader to \cite{12}. In the sequel, the symbol $\{u(x),v(x)\},$ $x\in\mathbb R$ stands for a PFSS of \eqref{2.1}.

    \begin{thm}\label{thm2.1.2}\cite{3}
    Equation \eqref{2.1} has a PFSS if
    \begin{equation}\label{2.5}
    \int_{-\infty}^0q(t)dt>0,\qquad \int_0^\infty q(t)dt>0.
        \end{equation}
    \end{thm}

    \begin{thm}\label{thm2.1.3}\cite{3},\cite{11},\cite[\S19.53]{18}.
    Suppose that equation \eqref{2.1} has a PFSS. Denote
    \begin{equation}\label{2.6}
     \rho(x)=u(x)v(x),\quad x\in\mathbb R.
\end{equation}
    Then one has the Davies-Harrell formulas:
       \begin{equation}\label{2.7}
    u(x)=\sqrt{\rho(x)}\exp\left(-\frac{1}{2}\int_{x_0}^x\frac{d\xi}{r(\xi)\rho(\xi)}\right),\ v(x)=\sqrt{\rho(x)}\exp\left(\frac{1}{2}\int_{x_0}^x\frac{d\xi}{r(\xi)\rho(\xi)}\right),\ x\in\mathbb R,
    \end{equation}
    where $x_0$ is a unique root of the equation
    $u(x)=v(x),$  $x\in\mathbb R.$
    \end{thm}

    \begin{cor}\label{cor2.1.4}\cite{8}
    We have the following relations for a PFSS of equation \eqref{2.1} and the functions $\rho(\cdot)$ (see \eqref{2.6}):
    \begin{equation}\label{2.8}
    \frac{u'(x)}{u(x)}=-\frac{1-r(x)\rho'(x)}{2r(x)\rho(x)},\qquad \frac{v'(x)}{v(x)}=\frac{1+r(x)\rho'(x)}{2r(x)\rho(x)},\qquad  x\in\mathbb R,
    \end{equation}
     \begin{equation}\label{2.9}
    r(x)|\rho'(x)|<1,\qquad x\in\mathbb R.
    \end{equation}
    \end{cor}

    \begin{remark}\label{rem2.1.5}
    In a particular case, formulas \eqref{2.7} were obtained by N. Abel (see \cite[19.53]{18}). Under the conditions \eqref{1.2} and $r\equiv 1,$ they were proven in \cite{11}. In the present statement, \thmref{thm2.1.3} was proven in \cite{3}.
    \end{remark}

    \subsection{Hartman-Wintner problem}

    In this section, we consider the equations
    \begin{gather}\label{2.10}
    (r(x)y'(x))'=q(x)y(x),\qquad x\in\mathbb R,
    \\ \label{2.11}
    (r(x)z'(x))'=q_1(x)z(x),\qquad x\in\mathbb R,
    \end{gather}
    where the functions $r(x),$ $q(x)$ and $q_1(x)$ are real and continuous for $x\in\mathbb R$ and, in addition, $r(x)>0$ for $x\in\mathbb R.$ We also assume that equation \eqref{2.11} has a PFSS $\{u_1(x),v_1(x)\},$ $x\in\mathbb R.$

    \begin{defn}\label{defn2.21} \cite{2,12}
    We say that the Hartman-Wintner problem for equations \eqref{2.10} and \eqref{2.11} is solvable as $x\to\infty$ (as $x\to-\infty)$ if there is a FSS $\{\hat u(x),\hat v(x)\},$ $x\in[0,\infty)$ $(\{\tilde u(x),\tilde v(x)\}$ $x\in(-\infty,0])$ of equation  \eqref{2.10} such that
    \begin{equation}\label{2.12}
    \lim_{x\to\infty}\frac{\hat u(x)}{u_1(x)}=\lim_{x\to\infty}\frac{\hat v(x)}{v_1(x)}=1\qquad\left(\lim_{x\to-\infty}\frac{\tilde u(x)}{u_1(x)}=\lim_{x\to-\infty}\frac{\tilde v(x)}{v_1(x)}=1\right),
    \end{equation}
    \begin{alignat}{2}
  & \frac{\hat u'(x)}{\hat u(x)}-\frac{u_1'(x)}{u_1(x)}=o\left(\frac{1}{r(x)u_1(x)v_1(x)}\right)\qquad&&\text{as}\quad x\to\infty,
  \label{2.13} \\
   & \bigg( \frac{\tilde u'(x)}{\tilde u(x)}-\frac{u_1'(x)}{u_1(x)}=o\left(\frac{1}{r(x)u_1(x)v_1(x)}\right)\qquad
   &&\text{as}\quad x\to-\infty\bigg),
  \nonumber  \end{alignat}
    \begin{alignat}{2}
  & \frac{\hat v'(x)}{\hat v(x)}-\frac{v_1'(x)}{v_1(x)}=o\left(\frac{1}{r(x)u_1(x)v_1(x)}\right)\qquad&&\text{as}\quad x\to\infty,
  \label{2.14} \\
   & \bigg( \frac{\tilde v'(x)}{\tilde v(x)}-\frac{v_1'(x)}{v_1(x)}=o\left(\frac{1}{r(x)u_1(x)v_1(x)}\right)\qquad
   &&\text{as}\quad x\to-\infty\bigg).
  \nonumber  \end{alignat}
    \end{defn}

    \begin{remark}\label{rem2.2.2}
    Here, taking into account the goals of the present work, we slightly restrict the statement of the Hartman-Wintner problem compared to the original papers (see \cite{2}, \cite{12}). For brevity, below we refer to the Hartman-Wintner problem as problem \eqref{2.12}-\eqref{2.14}.
    \end{remark}

    We need the following notation:
    \begin{gather}\label{2.15}
    \rho_1(x)=u_1(x)v_1(x),\quad x\in\mathbb R,
    \\ \label{2.16}
    (\Delta q)(x)=q(x)-q_1(x),\quad x\in\mathbb R,
\\ \label{2.17}
    I^{(-)}(x)=\int_{-\infty}^x(\Delta q)(t)p_1(t)dt,\quad I^{(+)}(x)=\int_x^\infty(\Delta q)(t)\rho_1(t)dt,\quad x\in\mathbb R.
    \end{gather}

 \begin{thm}\label{thm2.2.3} (\cite{12}; see also \cite{2})
 Problem \eqref{2.12}-\eqref{2.14} for equations \eqref{2.10} and \eqref{2.11} is solvable as $x\to\infty$ (as $x\to-\infty)$ if the integral $I^{(+)}(0)$ $(I^{(-)}(0))$ absolutely converges.
 \end{thm}

  \begin{thm}\label{thm2.2.4}   \cite{2})
 Problem \eqref{2.12}-\eqref{2.14} for equations \eqref{2.10} and \eqref{2.11} is solvable as $x\to\infty$ (as $x\to-\infty)$ if the integral $I^{(+)}(0)$ $(I^{(-)}(0))$   converges at least conditionally, and the following inequality holds:
 \begin{equation}\label{2.18}
 \int_0^\infty\frac{I^{(+)}(x)^2}{r(x)\rho_1(x)}dx<\infty\quad
 \left(\int_{-\infty}^0\frac{I^{(-)}(x)^2}{r(x)\rho_1(x)} dx<\infty\right).
 \end{equation}
 \end{thm}
 \subsection{Some facts on the correct solvability of equation \eqref{1.1}}

 \begin{thm}\label{thm2.3.1} \cite{6} Suppose that
 \begin{equation}\label{2.19}
 \int_{-\infty}^0\frac{dt}{r(t)}=\int_0^\infty\frac{dt}{r(t)}=\infty
 \end{equation}
 and equation \eqref{1.1} is correctly solvable in $L_p,$ $p\in[1,\infty).$
 Then
 \begin{equation}\label{2.20}
 \int_{-\infty}^x q(t)dt>0,\qquad \int_x^\infty q(t)dt>0\qquad\text{for every}\quad x\in\mathbb R,
 \end{equation}
 and therefore  equation \eqref{2.1} has a PFSS (see \thmref{thm2.1.2}).

 \end{thm}

 Throughout this section, we assume that \eqref{2.20} holds. This standing assumption does not appear in the statements. From \eqref{2.20} and \thmref{thm2.1.2} it follows that equation \eqref{2.1} has a PFSS $\{u(x),v(x)\},$ $x\in\mathbb R.$ Denote
 \begin{equation}\label{2.21}
 G(x,t)=\begin{cases} u(x)\cdot v(t),\ & x\ge t\\
 u(t)\cdot v(x),\ & x\le t\end{cases}    ,\qquad x,t\in\mathbb R;
 \end{equation}
 \begin{equation}\label{2.22}
 (Gf)(x)=\int_{-\infty}^\infty G(x,t)f(t)dt,\quad x\in\mathbb R,\  f\in L_p;
 \end{equation}
  \begin{equation}\label{2.23}
 (G_1f)(x)=u(x)\int_{-\infty}^x v(t)f(t)dt,\quad (G_2f)(x)=v(x)\int_x^\infty u(t)f(t)dt,\quad x\in\mathbb R,\  f\in L_p.
 \end{equation}

 Here $G(\cdot,\cdot)$ is the Green function, and $G:L_p\to L_p$, $p\in[1,\infty)$ is the Green operator of equation \eqref{1.1}.

 \begin{lem}\label{lem2.3.2} \cite{6} The following relations hold:
 \begin{equation}\label{2.24}
 (Gf)(x)=(G_1f)(x)+(G_1f)(x),\quad x\in\mathbb R, \ f\in L_p,\ p\in[1,\infty),
 \end{equation}
  \begin{equation}\label{2.25}
\frac{1}{2}\big(\|G_1\|_{p\to p}+\|G_2\|_{p\to p}\big)\le \|G\|_{p\to p}\le \|G_1\|_{p\to p}+\|G_2\|_{p\to p}, \ p\in[1,\infty).
 \end{equation}
 \end{lem}

  \begin{thm}\label{thm2.3.3}  \cite{6}
Let $p\in[1,\infty)$, and suppose that \eqref{2.20} holds. Then equation \eqref{1.1} is correctly solvable in $L_p$ if and only if the operator $G:L_p\to L_p$ is bounded.
 \end{thm}

 \begin{cor}\label{cor2.3.4}\cite{6}
 Let $p\in [1,\infty)$ and suppose that $\|G\|_{p\to p}<\infty.$ Then for $f\in L_p,$ the solution $y\in L_p$ of \eqref{1.1} is of the form
 \begin{equation}\label{2.26}
 y(x)=(Gf)(x)=\int_{-\infty}^\infty G(x,t)f(t)dt,\quad x\in \mathbb R.
 \end{equation}
 \end{cor}

 \begin{lem}\label{lem2.3.5} \cite{3}
 Under condition \eqref{1.3}, for any $x\in \mathbb R$ each of the equations
 \begin{equation}\label{2.27}
 \int_{x-d}^x\frac{dt}{r(t)}\cdot \int_{x-d}^x q(t)dt=1,\qquad \int_x^{x+d}\frac{dt}{r(t)}\cdot\int_x^{x+d} q(t)dt=1
 \end{equation}
 in $d\ge0$ has a unique finite positive solution. Denote them by $d_1(x)$ and $d_2(x)$, respectively, and introduce the functions
 \begin{equation}\label{2.28}
 \varphi(x)=\int_{x-d_1(x)}^x\frac{dt}{r(t)},\quad \psi(x)=\int_x^{x+d_2(x)}\frac{dt}{r(t)},\quad x\in\mathbb R,
 \end{equation}
  \begin{equation}\label{2.29}
h(x)=\frac{\varphi(x)\cdot\psi(x)}{\varphi(x)+\psi(x)}\ \left(=\left(\int_{x-d_1(x)}^{x+d_2(x)}q(t)dt\right)^{-1}\right),
\quad x\in\mathbb R.
 \end{equation}
 Further, for any $x\in\mathbb R,$ the equation
  \begin{equation}\label{2.30}
\int_{x-d}^{x+d}\frac{dt}{r(t)h(t)}=1
 \end{equation}
  in $d\ge0$ has a unique finite positive solution. Denote it by $d(x),$ $x\in\mathbb R.$ The function $d(x)$ is continuous for $x\in\mathbb R.$
  \end{lem}

  \begin{remark}\label{rem2.3.6}
  {}From \thmref{thm2.1.2} it follows that under condition \eqref{2.20}, equation \eqref{2.1} has a PFSS $\{u(x),v(x)\},$ $x\in\mathbb R,$ and therefore the function $\rho(x),$ $x\in\mathbb R$, is well-defined (see \eqref{2.6}).
  \end{remark}

  \begin{thm}\label{thm2.3.7}\cite{3}
  Under condition \eqref{1.3}, we have
  \begin{equation}\label{2.31}
  2^{-1}h(x)\le\rho(x)\le 2h(x),\quad x\in\mathbb R.
  \end{equation}
  \end{thm}

  \begin{remark}\label{rem2.3.8}
  Two-sided, sharp by order, a priori estimates for the function $\rho(x),$ $x\in\mathbb R$, were first obtained in \cite{17} (under some additional requirements for $r$ and $q$) by M.O.~Otelbaev. Therefore, all estimates of type \eqref{2.31} are called Otelbaev inequalities. Note that in \cite{17} Otelbaev used a more complicated auxiliary function than $h(x),$ $x\in\mathbb R.$
  \end{remark}

  \begin{thm}\label{thm2.3.9} \cite[\S4 (Theorem 4.1)]{6}
  Let $p\in(1,\infty),$ and suppose that \eqref{1.3} holds.
Then equation \eqref{1.1} is correctly solvable in $L_p$ if and only if $B<\infty$ where
 \begin{equation}\label{2.32}
  B=\sup_{x\in\mathbb R}(h(x)d(x)).
  \end{equation}
\end{thm}

\subsection{Two theorems on integral operators}

\begin{thm}\label{thm2.4.1}\cite[Ch. V, \S2, no.5]{13}
Let $-\infty\le a<b\le\infty,$ let $K(s,t)$ be a continuous function in $s,t\in[a,b],$ and let $K$ be an integral operator of the form
\begin{equation}\label{2.33}
(Kf)(t)=\int_a^bK(s,t,)f(s)ds,\ t\in[a,b].
\end{equation}
Then
\begin{equation}\label{2.34}
\|K\|_{L_1(a,b)\to L_1(a,b)}=\sup_{s\in[a,b]}\int_a^b|K(s,t)|dt.
\end{equation}
\end{thm}

\begin{thm}\label{thm2.4.2}\cite{14}
Let $\mu(\cdot)$ and $\theta(\cdot)$ be positive continuous functions in $\mathbb R$, and let $K^{(+)}(K^{(-)}$ be an integral operator of the form
\begin{equation}\label{2.35}
(K^{(+)}f)(t)=\mu(t)\int_t^\infty\theta(\xi)f(\xi)d\xi,\ t\in\mathbb R;\quad ((K^{(-)}f)(t)=\mu(t)\int_{-\infty}^t\theta(\xi) f(\xi)d\xi,\ t\in\mathbb R).
\end{equation}
Then for $p\in(1,\infty),$ the operator $K^{(+)}:L_p\to L_p$ $((K^{(-)}:L_p\to L_p)$ is bounded if and only if $H_p^{(+)}<\infty$ $(H_p^{(-)}<\infty).$ Here
\begin{equation}\label{2.36}
H_p^{(+)}=\sup_{x\in\mathbb R}H_p^{(+)}(x)\quad (H_p^{(-)}=\sup_{x\in\mathbb R}H_p^{(-)}(x)),
\end{equation}
where
\begin{gather}
H_p^{(+)}(x) =\left(\int_{-\infty}^x\mu(t)^pdt\right)^{1/p}
\left(\int_x^\infty\theta(t)
^{p'}dt\right)^{1/p'}
\label{2.37}\\
 \left(H_p^{(-)}(x)=\left(\int_{-\infty}^x\theta(t)^{p'}dt\right)^{1/p'}\left(\int_x^\infty\mu(t)^{p}dt\right)^{1/p}
\right).\nonumber
\end{gather}
In addition, we have the inequalities
\begin{equation}\label{2.38}
H_p^{(+)}\le\|K^{(+)}\|_{p\to p}\le(p)^{1/p}(p')^{1/p'}H_p^{(+)}\quad (H_p^{(-)}\le\|K^{(-)}\|_{p\to p}\le(p)^{1/p}(p')^{1/p} H_p^{(-)}).
\end{equation}
\end{thm}

\subsection{On Otelbaev's coverings of the real semi-axis}

\begin{defn}\label{defn2.5.1} \cite{3,5,15}
Suppose we are given $x\in\mathbb R,$ a positive continuous function $\varkappa(x),$ $x\in\mathbb R,$ and a sequence $\{x_n\}_{n\in N'},$ $N'=\{\pm1,\pm 2,\dots\}\, .$ Consider the segments $\Delta_n=x_n\pm\varkappa(x_n),$ $n\in N'.$ We say that the sequence of segments $\{\Delta_n\}_{n=1}^\infty$ $(\{\Delta_n\}_{n=-\infty}^{-1})$ forms an $\mathbb R(x,\varkappa(\cdot))$-covering of $[x,\infty)$ (resp. an $\mathbb R(x,\varkappa(\cdot))$-covering of $(-\infty,x])$ if the following conditions hold:
\begin{enumerate}
\item[1)] $\Delta_n^{(+)}=\Delta_{n+1}^{(-)}$ for $n\ge 1$ (resp. $\Delta_{n-1}^{(+)}=\Delta_n^{(-)}$ for $n\le -1);$
    \item[2)] $\Delta_1^{(-)} =x$\quad (resp. $\Delta_{-1}^{(+)} =x);$
     \item[3)] $\cupl\limits_{n\ge 1}^\infty\Delta_n=[x,\infty)$\quad (resp.  $\cupl\limits_{n\le -1}^\infty\Delta_n=(-\infty,x]).$
        \end{enumerate}
\end{defn}

\begin{lem}\label{lem2.5.2}\cite{5,15}
    Suppose that a positive continuous function $\varkappa(x),$ $x\in\mathbb R,$ satisfies the relations
    \begin{equation}\label{2.39}
    \lim_{x\to\infty}(x-\varkappa(x))=\infty,\quad \lim_{x\to-\infty}(x+\varkappa(x))=-\infty.
    \end{equation}
    Then for any $x\in\mathbb R$ there exists an $\mathbb R(x,\varkappa(\cdot))$-covering of $[x,\infty)$ (resp. an $\mathbb R(x,\varkappa(\cdot))$-covering of $(-\infty,x]).$
    \end{lem}

    \section{Results and Additional Comments}

    The following assertion can be viewed as a complement to \thmref{thm2.1.2}.

    \begin{theorem}\label{thm3.1}
    Let $1/r(\cdot)\in L_1.$ Then the functions
    \begin{equation}\label{3.1}
    u(x)=\frac{1}{\sqrt{w_0}}\int_x^\infty\frac{dt}{r(t)},\quad v(x)=\frac{1}{\sqrt{w_0}}\int_{-\infty}^x\frac{dt}{r(t)},\quad x\in\mathbb R,\quad w_0=\int_{-\infty}^\infty\frac{dt}{r(t)}
    \end{equation}
    form a PFSS $\{u(x),v(x)\}$, $x\in\mathbb R$, of the equation
    \begin{equation}\label{3.2}
    (r(x)z'(x))'=0,\quad x\in\mathbb R.
    \end{equation}
    \end{theorem}

    Definition \ref{defn2.1.1} can be complemented by \thmref{thm3.2} and \corref{cor3.3}.
    \begin{theorem}\label{thm3.2}
    Suppose that equation \eqref{2.1} has a PFSS $\{u(x),v(x)\},$ $x\in\mathbb R.$ Then any other PFSS $\{u_1(x),v_1(x)\},$ $x\in \mathbb R,$ of this equation is of the form
    \begin{equation}\label{3.3}
    u_1(x)=\alpha u(x),\quad v_1(x)=\alpha^{-1}v(x),\quad x\in\mathbb R.
    \end{equation}
    \end{theorem}

    Here $\alpha$ is an arbitrary fixed positive constant.

    \begin{corollary}\label{cor3.3}
    If equation \eqref{2.1} has a PFSS $\{u(x),v(x)\},$ $x\in\mathbb R,$ then the function $\rho(x),$ $x\in\mathbb R$ (see \eqref{2.6}) does not depend on the choice of a PFSS.
    \end{corollary}

    \begin{remark}\label{rem3.4}
    Therefore, the function $\rho(x),$ $x\in\mathbb R,$ is a well-defined implicit function in the coefficients of equation \eqref{2.1}. It is instructive to compare this fact with formulas \eqref{2.7} and the Davies-Harrell representation of the Green function $G(\cdot,\cdot)$ (see \eqref{2.21}, \cite{11} and \cite{3}):
    \begin{equation}\label{3.4}
    G(x,t)=\sqrt{\rho(x)\rho(t)}\exp\left(-\frac{1}{2}\left|\int_x^t\frac{d\xi}{r(\xi)\rho(\xi)}\right|\right),\quad x,t\in\mathbb R.
    \end{equation}
    \end{remark}

Throughout the sequel, we call $\rho(x),$ $x\in\mathbb R,$ the generating function of the PFSS of equation \eqref{2.1}.

\begin{theorem}\label{thm3.5}
Suppose that
\begin{equation}\label{3.5}
\int_{-\infty}^0\frac{dt}{r(t)}=\int_0^\infty\frac{dt}{r(t)}=\infty,\quad q(x)=0\quad\text{for}\ x\in\mathbb R.
\end{equation}
Then equation \eqref{3.2} has no PFSS.
\end{theorem}

    In the remainder of this section, we assume that condition \eqref{1.2} holds and equation \eqref{2.1} has a PFSS $\{u(x),v(x)\},$ $x\in\mathbb R.$ In the absence of a special need, we do not include these requirements in the statements. Note that we do not assume that condition \eqref{1.3} holds.

    \begin{theorem}\label{thm3.6}
    Equation \eqref{2.1} has no solutions $y(\cdot)\in L_p,$ $p\in[1,\infty)$ except for $y(x)\equiv0,$ $x\in\mathbb R.$
    \end{theorem}

    \begin{theorem}\label{thm3.7}
    For $p\in[1,\infty),$ equation \eqref{1.1} is correctly solvable in $L_p$ if and only if the Green operator $G:L_p\to L_p$ is bounded (see \eqref{2.22}).
    \end{theorem}

    \begin{corollary}\label{cor3.8}
    Let $p\in[1,\infty)$, and suppose that equation \eqref{1.1} is correctly solvable in $L_p.$
   Then for any $f\in L_p$, the solution $y\in L_p$ of \eqref{1.1} is of the form
   \begin{equation}\label{3.6}
   y(x)=(Gf)(x)=\int_{-\infty}^\infty G(x,t)f(t)dt,\quad x\in\mathbb R.
   \end{equation}
    \end{corollary}

    \thmref{thm3.7} has the following advantage compared to \thmref{thm2.3.3}: the requirement of the existence of a PFSS of equation \eqref{2.1} is weaker than condition \eqref{2.20} of \thmref{thm2.3.3} which guarantees this requirement. In particular, \thmref{thm3.7} is applicable in the case \eqref{3.1}, which is not covered by \thmref{thm2.3.3}.

    Now let us introduce a new auxiliary function $s(x),$ $x\in\mathbb R,$ an analogue of the function $d(x),$ $x\in\mathbb R,$ (see \lemref{lem2.3.5}). To this end, fix $x\in\mathbb R$ and consider the following equation in $s\ge0:$
    \begin{equation}\label{3.7}
    \int_{x-s}^{x+s}\frac{dt}{r(t)\rho(t)}=1.
    \end{equation}

    \begin{lemma}\label{lem3.9}
    For any $x\in\mathbb R$, there exists a unique solution of \eqref{3.7} in $s\ge0;$ denote it by $s(x).$ The function $s(x)$ is positive, continuous, and differentiable almost everywhere in $\mathbb R.$ In addition, the following relations hold:
    \begin{alignat}{2}
    & |s(x+t)-s(x)|\le |t| &&\quad\text{for}\quad |t|\le s(x),\ x\in\mathbb R,\label{3.8}\\
     & |s'(x)|<1 &&\quad\text{for almost all}\quad   x\in\mathbb R,\label{3.9}
\end{alignat}
\begin{equation}\label{3.10}
\lim_{x\to-\infty}(x+s(x))=-\infty,\qquad \lim_{x\to\infty}(x-s(x))=\infty,
\end{equation}
\begin{equation}\label{3.11}
 0<s(x)\le|x|\qquad\text{for}\quad |x|\gg1.
\end{equation}
Finally, there is a constant $c\in[1,\infty)$ such that for all $x\in\mathbb R,$ we have the inequality
\begin{equation}\label{3.12}
 s(x)\le c(1+|x|).
\end{equation}
\end{lemma}

\begin{remark}\label{rem3.10}
Functions similar to $s(x),$ $x\in\mathbb R$ (see also the functions $d(x)$, $x\in\mathbb R,$ from \lemref{lem2.3.5}) were introduced and used in various areas of analysis by  Otelbaev (see \cite{15}). The fact that one of such functions is Lipschitz was first proved by K.T.~Mynbaev (see \cite{15}).
\end{remark}

\begin{lemma}\label{lem3.11}
For $t\in[x-s(x),x+s(x)],$ $x\in\mathbb R,$ we have the following estimates:
\begin{equation}\label{3.13}
e^{-1}\rho(x)\le \rho(t)\le e\rho(x),\quad e=\exp(1),
\end{equation}
\begin{equation}\label{3.14}
e^{-1}u(x)\le u(t)\le eu(x),\quad e^{-1} v(x)\le v(t)\le ev(v).
\end{equation}
\end{lemma}

\begin{theorem}\label{thm3.12}
For $p\in(1,\infty),$ equation \eqref{1.1} is correctly solvable in $L_p$ if and only if $\mathcal D<\infty.$ Here
\begin{equation}\label{3.15}
\mathcal D=\sup_{x\in\mathbb R}(\rho(x)s(x)).
\end{equation}
In particular, the following inequalities hold (see \eqref{2.22}, \eqref{2.23}):
\begin{equation}\label{3.16}
c^{-1}\mathcal D\le\|G_1\|_{p\to p};\quad \|G_2\|_{p\to p};\quad \|G\|_{p\to p}\le c\mathcal D.
\end{equation}
\end{theorem}

It is easy to see that under condition \eqref{1.3}, Theorems \ref{thm2.3.9} and \ref{thm3.12} are equivalent. However, \thmref{3.12} is more general than \thmref{thm2.3.9}. For example, it is applicable in the case $1/r\in L_1,$ $q\equiv0$ whereas \thmref{thm2.3.9} is obviously not applicable because of condition \eqref{2.20}.

\begin{corollary}\label{cor3.13}
For $p\in(1,\infty)$, equation \eqref{1.1} is correctly solvable in $L_p$ if any of the following inequalities holds:
\begin{align}
\sigma_1 &= \sup_{x\in\mathbb R}(r(x)\rho(x)^2)<\infty, \label{3.17}\\
\sigma_2 &= \sup_{x\in\mathbb R}(\rho(x)|x|)<\infty, \label{3.18}\\
\sigma_3 &= \inf_{x\in\mathbb R}\left(\frac{1}{2s(x)}\int_{x-s(x)}^{x+s(x)}q(\xi)d\xi\right)>0,  \label{3.19}\\
 \sigma_4 &=  \inf_{x\in\mathbb R} q(x)>0, \label{3.20}\\
\sigma_5 &= \sup_{x\in\mathbb R}\left[r(x)\left(\int_{-\infty}^x\frac{dt}{r(t)}\right)^2
 \left(\int_x^\infty\frac{dt}{r(t)}\right)^2\right]<\infty. \label{3.21}
 \end{align}
 \end{corollary}

 Note that the fact that condition \eqref{3.20} is sufficient for the correct solvability of \eqref{1.1} in $L_p$ was established in \cite{16},

 \begin{corollary}\label{cor3.14}
 Suppose that the following two conditions hold:
 \begin{equation}\label{3.22}
 1)\ r_0\doe\inf_{x\in\mathbb R}(r(x))<\infty;\qquad 2)\ \lim_{|x|\to\infty} r(x)\rho(x)=\infty.
 \end{equation}
 Then for $p\in(1,\infty),$ equation \eqref{1.1} is not correctly solvable in $L_p.$
  \end{corollary}

  In connection with \corref{cor3.14}, note that if the function $q(x),$ $x\in\mathbb R$ is finitary and the function $r(x),$ $x\in\mathbb R$, satisfies condition \eqref{2.19}, then for $p\in(1,\infty)$ equation \eqref{1.1} is not solvable in $L_p$ (see \thmref{thm2.3.1}).

  \begin{definition}\label{defn3.15}
  Let functions $f(x)$ and $g(x)$ be defined, continuous and positive for $x\in(a,b),$ $-\infty\le a<b\le \infty.$
  We say that these functions are weakly equivalent for $x\in(a,b)$ and write $f(x)\asymp g(x),$ $x\in (a,b)$ if there is a constant $c\in[1,\infty)$ such that for all $x\in(a,b)$, we have the inequalities
  \begin{equation}\label{3.23}
 c^{-1}f(x)\le g(x)\le cf(x).
 \end{equation}
 \end{definition}

 In the next definition, together with \eqref{1.1}, we consider the equation
   \begin{equation}\label{3.24}
 -(r(x)y'(x))'+q_1(x)y(x)=f(x),\quad x\in\mathbb R
 \end{equation}
 under the condition
    \begin{equation}\label{3.25}
 r(\cdot)>0,\quad q_1(\cdot)\ge0,\quad\frac{1}{r}\in L_1^{\loc}(\mathbb R),\quad q_1\in L_1^{\loc}(\mathbb R.
 \end{equation}
 We also assume, without mention in the statements, that equation \eqref{2.11} like equation \eqref{3.24} satisfies condition \eqref{3.25} and has a PFSS $\{u_1(x),v_1(x)\},$ $x\in\mathbb R.$
   \begin{definition}\label{defn3.16}
   We say that for equations \eqref{1.1} and \eqref{3.24} problems I)-II), $p\in(1,\infty),$ are equivalent if from the correct solvability of one of them follows the correct
    solvability of the other.
   \end{definition}

   \begin{theorem}\label{thm3.17}
   For equations \eqref{1.1} and \eqref{3.24}, problems I)-II), $p\in(1,\infty)$ are equivalent (not equivalent) if for $x\in\mathbb R$ the functions $\rho(x)$ and $\rho_1(x),$ generating PFSS of equations \eqref{2.1} and \eqref{2.11} are (are not) weakly equivalent.
   \end{theorem}

   \begin{theorem}\label{thm3.18}
   Suppose we are given functions $r(x),$ $q(x),$ $q_1(x)$, continuous for $x\in\mathbb R,$ such that $r(x)>0$ and the following two conditions hold:
   \begin{enumerate}
   \item[1)]  each of the equations \eqref{2.1} and \eqref{2.11} have a PFSS $\{u(x),v(x)\},$ $x\in\mathbb R$, and $\{u_1(x),v_1(x)\},$ $x\in\mathbb R,$ respectively;
       \item[2)] problems \eqref{2.12}-\eqref{2.14} for equations \eqref{2.1} and \eqref{2.11} are solvable as $x\to\pm\infty.$
           \end{enumerate}
           Then for $x\in\mathbb R$ the functions $\rho(x)$ and $\rho_1(x),$ generating PFSS of equations \eqref{2.1} and \eqref{2.11}, are weakly equivalent.
    \end{theorem}

     \begin{theorem}\label{thm3.19}
     Suppose we are given functions $r(x),$ and $q(x),$   continuous for $x\in\mathbb R,$ such that $r(x)>0$ and $q(x)>0$ for $x\in\mathbb R$ and $1/r\in L_1.$ Then, assuming that we have the inequalities
       \begin{equation}\label{3.26}
 -\int_{-\infty}^0q(x)\left(\int_{-\infty}^x\frac{dt}{r(t)}\right)dx<\infty,\quad \left(\int_0^\infty q(x)\int_x^\infty\frac{dt}{r(t)}\right)dx<\infty,
 \end{equation}
 Then problems I)-II), $p\in(1,\infty)$ for equations \eqref{1.1} and
  \begin{equation}\label{3.27}
 (r(x)y'(x))'=f(x),\quad x\in\mathbb R
 \end{equation}
 are equivalent. In particular, these problems are equivalent if instead of \eqref{3.26} we require the inclusion $q(\cdot)\in L_1.$
         \end{theorem}

         Let us make some comments on the last three assertions. First note that   applying \thmref{thm3.17}    is based on Theorem \ref{thm2.3.7}   in the case \eqref{1.3} and on Theorems \ref{thm3.18} and \ref{thm2.2.3}  in the case \eqref{1.4}. Furthermore, \thmref{thm3.17} is significant in the study of problem I)-II) for two reasons. First, it is equally relevant in both the cases \eqref{1.3} and \eqref{1.4}. Second, \thmref{thm3.17} allows one to control the transition from the original problem I)-II) for equation \eqref{1.1} to a similar equivalent problem for the simple (model) equation \eqref{3.24}. Here the relation $\rho(x)\asymp\rho_1(x),$ $x\in\mathbb R$ serves as a criterion for such a transition to be correct. The conditions for checking this relation are given in \thmref{thm2.3.7} in the case \eqref{1.3}, whereas in the case \eqref{1.4} such a verification is especially simple due to Theorems \ref{thm3.18} and \ref{thm2.2.3} (see the proof of \thmref{thm3.19}). Finally, we describe the conceptual difference between the case \eqref{1.3} and \eqref{1.4}. It is based on the fact that in the case \eqref{1.4}, for equation \eqref{1.1}, regardless of its coefficient $q(\cdot)\in L_1,$ (within the framework of problem I)-II)) one can immediately find  the simplest appropriate model equation  \eqref{3.27} (see \thmref{thm3.19}), whereas (within the same problem I)-II))in the case \eqref{1.3}, the choice of an appropriate model equation \eqref{3.24}  is a nontrivial problem in its own right.

         To conclude this section, note that for the sake of keeping the present paper with a limited amount of pages, we do not study a priori estimates for the function $s(x), $ $x\in\mathbb R$ (see \eqref{3.7} and \thmref{thm3.12}). Therefore, in the study of the example in \S5 we are forced to use special trick and the assertions of \corref{cor3.13}. We plan to study the function $s(x), $ $x\in\mathbb R$  and a direct application of \thmref{thm3.12} in a subsequent paper.

         \section{Proofs and Comments}

         First we note some special features of this section. Our earlier study of equations \eqref{1.1} and \eqref{2.1} was based on the following   (see \cite{3,4,6,7,8,9}):
  \begin{enumerate}\item[1)] Definition \ref{defn2.1.1} of a PFSS of equation \eqref{2.1} and the condition for the their existence (see \thmref{thm2.1.2});
 \item[2)] Davies-Harrell formulas \eqref{2.7} and \eqref{3.4} (see \cite{3} for a simple proof);
     \item[3)] a priori estimates for the function $\rho(x), $ $x\in\mathbb R,$ generating a PFSS of equation \eqref{2.1} (see \cite{8,9}).
        \end{enumerate}

Items 1) and 3) require some comments (in regard to item 2), see Remark \ref{rem2.1.5} and \thmref{thm2.1.3}). We start with 1). The properties of solutions of equation \eqref{2.1} have been studied in many papers (see, e.g., \cite{12}), and below one can see significant overlaps with them. However, the present paper is the first instance of viewing Definition \ref{defn2.1.1} of a PFSS as a basic concept of the theory of equation \eqref{2.1} and, more significantly, it contains the first proof of all the main properties of PFSS based only on Definition \ref{defn2.1.1} and Davies-Harrell formulas \eqref{2.7} (which follow from Definition \ref{defn2.1.1}, see  \cite{3}.

Since we can easily locate the cases where our assertions overlap with the known ones (see \cite{12}), for the sake of brevity we do not mention or comment on them. Now we go to 3). Our main results are Theorems \ref{thm3.7}, \ref{thm3.12}, \ref{thm3.17} and \ref{thm3.19} and their corollaries. The proofs of Theorems \ref{thm3.7} and \ref{thm3.12} follow along the lines of \cite{6}. Therefore, in the cases where the relevant part of the proof only relies on the properties of PFSS of equation \eqref{1.1} and thus mimics the corresponding part of \cite{6}, we only state the assertion and refer the reader to \cite{6}. However, for the other theorems, we present complete new and shorter proofs than found in~\cite{6}.

\begin{proof}[Proof of \thmref{thm3.1}]
A straightforward check shows that in the case \eqref{3.1} the functions $\{u(x),v(x)\},$ $x\in\mathbb R,$ form a PFSS of equation \eqref{3.2}.
\end{proof}

\begin{proof}[Proof of \thmref{thm3.2}]

\begin{lemma}\label{lem4.1} \cite{3}
We have the inequalities
\begin{equation}\label{4.1}
\int_{-\infty}^0\frac{dt}{r(t)\rho(t)}=\int_0^\infty\frac{dt}{r(t)\rho(t)}=\infty.
\end{equation}
\end{lemma}
\begin{proof}
The following relations are deduced from \eqref{2.4}, \eqref{2.6} and \eqref{2.7}:
\begin{gather*}
\exp\left(-\int_{x_0}^x\frac{dt}{r(t)\rho(t)}\right)=\frac{u(x)}{v(x)}\to0\quad\text{as}\quad x\to\infty\quad \Rightarrow\quad\int_0^\infty\frac{dt}{r(t) \rho(t)}=\infty;\\
\exp\left(-\int_{x}^{x_0}\frac{dt}{r(t)\rho(t)}\right)=\frac{v(x)}{u(x)}\to0\quad\text{as}\quad x\to-\infty\quad \Rightarrow\quad\int^0_{-\infty}\frac{dt}{r(t)\rho(t)}=\infty.
\end{gather*}
\end{proof}

\begin{lemma}\label{lem4.2}
For a PFSS of equation \eqref{2.1}, we have the equalities
\begin{equation}\label{4.2}
u(x)=v(x)\int_x^\infty\frac{dt}{r(t)v^2(t)},\quad v(x)=u(x)\int_{-\infty}^x\frac{dt}{r(t)u^2(t)},
\quad x\in\mathbb R.\end{equation}
\end{lemma}
\begin{proof}
 Both equalities in \eqref{4.2} are checked in the same way; using \eqref{2.7} and \eqref{4.1}; for example,
 \begin{align*}
 \int_x^\infty\frac{dt}{r(t)v^2(t)}&=\int_x^\infty\frac{1}{r(t)\rho(t)}
\exp\left(-\int_{x_0}^t\frac{d\xi}{r(\xi)\rho(\xi)}\right)dt=-\exp\left(-\int_{x_0}^t\frac{d\xi}{r(\xi)\rho(\xi) }\right)\Bigg|_x^\infty\\
&=\exp\left(-\int_{x_0}^x\frac{d\xi}{r(\xi)\rho(\xi)}\right)=\frac{u(x)}{v(x)},\quad x\in\mathbb R.
\end{align*}
\end{proof}

\begin{lemma}\label{lem4.3}
For a PFSS of equation \eqref{2.1}, we have the relations
\begin{equation}\label{4.3}
\int_{-\infty}^0\frac{dt}{r(t)u^2(t)}<\infty,\quad \int_0^\infty\frac{dt}{r(t)v^2(t)}<\infty;
\end{equation}
\begin{equation}\label{4.4}
\int_{-\infty}^0\frac{dt}{r(t)v^2(t)}= \int_0^\infty\frac{dt}{r(t)u^2(t)}=\infty.
\end{equation}
\end{lemma}

\begin{proof}
Both inequalities in \eqref{4.3} are checked in the same way. For example, the second inequality in \eqref{4.3} follows from \eqref{4.2} for $x=0.$ The equalities \eqref{4.4} are also deduced in the same way from \eqref{2.7} and \eqref{4.1}; for example,
$$\int_{x_0}^\infty\frac{dt}{r(t)u^2(x)}=\int_{x_0}^\infty\frac{1}{r(t)\rho(t)}
\left(\int_{x_0}^t\frac{d\xi}{r(\xi)\rho(\xi)}\right) dt\ge\int_{x_0}^\infty\frac{dt}{r(t)\rho(t)}=\infty\ \Rightarrow\ \eqref{4.4}.$$
\end{proof}

Let us now continue with the proof of the theorem. Let $\{u(x),v(x)\},$ $x\in\mathbb R$ be a PFSS of equation \eqref{2.1}, and let $\{u_1(x),v_1(x)\},$ $x\in\mathbb R,$ be another PFSS of the same equation. Then there are constants $\alpha$ in $\mathbb R$ and $\beta$ in $\mathbb R$ such that $0<|\alpha|+|\beta|<\infty$ and we have the equality
\begin{equation}\label{4.5}
u_1(x)=\alpha u(x)+\beta v(x),\quad x\in\mathbb R.
\end{equation}

The following relations are deduced from \eqref{2.2}, \eqref{2.4} and \eqref{4.5}:
\begin{equation}\label{4.6}
\lim_{x\to-\infty}\frac{u_1(x)}{u(x)}=\alpha+\beta\lim_{x\to-\infty}\frac{v(x)}{u(x)}=\alpha\quad\Rightarrow\quad \alpha\ge0,
\end{equation}
\begin{equation}\label{4.7}
\lim_{x\to\infty}\frac{u_1(x)}{v(x)}=\alpha \lim_{x\to\infty}\frac{u(x)}{v(x) }+\beta=\beta\quad\Rightarrow\quad \beta\ge0.
\end{equation}

Assume that $\beta>0.$ Denote (see \eqref{2.2})
\begin{equation}\label{4.8}
g(x)=\frac{u_1(x)}{v(x)},\quad x\in\mathbb R;\quad m=\inf_{x\ge0}g(x)\quad\Rightarrow\quad m\ge0.\quad \text{(see \eqref{2.2})}
\end{equation}
By \eqref{2.2} and \eqref{4.8}, we obtain the inequality
\begin{equation}\label{4.9}
g'(x)=\frac{u_1'(x)}{v(x)}-\frac{v'(x)}{v^2(x)}u_1(x)\le 0,\quad x\in\mathbb R\quad\Rightarrow\quad m\in[0,\infty).
\end{equation}

Let us show that $m=0.$ Otherwise, if $m>0,$ we obtain (see \eqref{4.8} and \eqref{4.4}):
$$\frac{1}{v(x)}\ge\frac{m}{u_1(x)},\quad x\in\mathbb R\quad\Rightarrow\quad \infty>\int_0^\infty \frac{dt}{r(t)v^2(t)}\ge m\int_0^\infty\frac{dt}{r(t)u_1^2(t)}=\infty,$$
a contradiction. Hence $m=0$ and therefore we have (see \eqref{4.9}):
$$\beta=\lim_{x\to\infty}\frac{u_1(x)}{v(x)} -\alpha\lim_{x\to\infty}\frac{u(x)}{v(x)}=0\quad\Rightarrow\quad \alpha>0.$$

From \lemref{lem4.3} it now follows that
$$v_1(x)=u_1(x)\int_{-\infty}^x\frac{dt}{r(t)u_1^2(t)}=\alpha u(x)\int_{-\infty}^x\frac{dt}{r(t)\alpha^2u^2(t)}=\frac{v(x)}{\alpha}.$$
\end{proof}

 \begin{proof}[Proof of \corref{cor3.3}]
 The assertion immediately follows from \eqref{3.3}.
 \end{proof}

 \begin{proof}[Proof of \thmref{thm3.5}]
Clearly, one can form a FSS of equation  \eqref{3.2} from, say, the following functions:
\begin{equation}\label{4.10}
z_1(x)\equiv 1,\quad z_2(x)=\int_0^x\frac{dt}{r(t)},\quad x\in\mathbb R.
\end{equation}

Assume to the contrary that under condition \eqref{3.5}, equation \eqref{3.2} has a PFSS $\{u(x),v(x)\},$ $x\in\mathbb R.$
Then there are constants $\alpha\in\mathbb R $ and $\beta\in\mathbb R$ $(0<|\alpha|+|\beta|<\infty)$ such that
\begin{equation}\label{4.11}
u(x)=\alpha z_1(x)+\beta z_2(x),\quad x\in\mathbb R.
\end{equation}

Cases 1)\ $\beta\ne0$ and \ 2)\ $\beta=0$\ will be considered separately. In the case 1), from \eqref{2.2} it follows that
\begin{equation}\label{4.12}
0\ge u'(x)=\frac{\beta}{r(x)},\quad x\in\mathbb R\quad \Rightarrow\quad \beta<0.
\end{equation}
Then, according to \eqref{4.4} and \eqref{4.12}, we get (see \eqref{3.5}):
$$u(x)=\alpha+\beta\int_0^x\frac{dt}{r(t)}<0,\quad\text{for}\quad x\gg 1,$$
which contradicts \eqref{2.2}, i.e., case 1) does not occur. Let us now consider case 2). Since $\{u(x),v(x)\},$ $x\in\mathbb R,$ is a PFSS of equation \eqref{3.2}, then
$\beta=0$ from \eqref{4.11} it follows that $\alpha>0$ (see \eqref{2.2}), and therefore, by \eqref{4.2} and \eqref{3.5}, we have
 $$\infty>v(x)=u(x)\int_{-\infty}^x\frac{dt}{r(t)u^2(t)}=\alpha\int_{-\infty}^x\frac{dt}{r(t)\alpha^2} =\frac{1}{\alpha}\int_{-\infty}^x\frac{dt}{r(t)}=\infty.$$
 The obtained contradiction shows that case 2) also does not occur.
 \end{proof}

 \begin{proof}[Proof of \thmref{thm3.6}]
 Assume to the contrary that there exists a solution $z(\cdot)\in L_p$ of \eqref{2.1} such that  $z(x)\not\equiv 0$ for $x\in\mathbb R.$ Then, since $\{u(x), v(x)\},$ $x\in\mathbb R$ is a PFSS of \eqref{2.1}, there are constants $\alpha\in\mathbb R$ and $\beta\in \mathbb R$ $(0<|\alpha|+|\beta|<\infty)$ such that
 \begin{equation}\label{4.13}
 z(x)=\alpha u(x)+\beta v(x),\quad x\in\mathbb R.
 \end{equation}

 Assume that $\beta\ne0.$ Then there is $x_1\gg1$ such that the following inequality holds (see \eqref{2.4}):
 $$\left|\frac{\alpha}{\beta}\right|\frac{u(x)}{v(x)}\le\frac{1}{2}\quad\text{for}\quad x\ge x_1;$$
 and therefore, according to \eqref{2.2}, we have
 \begin{align*}
 \infty&>\|z\|_p^p=\int_{-\infty}^\infty|z(x)|^pdx\ge\int_{x_1}^\infty|\alpha u(x)+\beta v(x)|^pdx\\
 &\ge|\beta|^pv(x_1)^p\int_{x_1}^\infty\left|1-
 \left|\frac{\alpha}{\beta}\right|\frac{u(x)}
 {v(x)}\right|^p dx\ge|\beta|^pv(x_1)^p\int_{x_1}^\infty \frac{dx}{2^p}=\infty
\end{align*}

The obtained contradiction shows that $\beta=0$ and therefore $\alpha\ne0$ (see \ref{4.13}). Once again, we apply \eqref{2.2} and get
$$\infty>\|z\|_p^p=|\alpha|^p\int_{-\infty}^\infty u(x)^pdx\ge|\alpha|^p\int_{-\infty}^0u(x)^pdx\ge|\alpha|^pu(o)^p
\int_{-\infty}^01dx=\infty,$$
a contradiction. Hence $\alpha=0$ and therefore $z(x)\equiv0$, $x\in\mathbb R, $ contradicting the original assumption.
\end{proof}

\begin{lemma}\label{lem4.4} \cite{6}
For the norms of the integral operators $G_1,G_2$ and $G$ (see \eqref{2.21}, \eqref{2.22} and \eqref{2.23}), we have the following relations:
\begin{equation}\label{4.14}
G=G_1+G_2;
\end{equation}
\begin{equation}\label{4.15}
\frac{|G_1\|_{p\to p}+\|G_2\|_{p\to p}}{2}\le \|G\|_{p\to p}\le \|G_1\|_{p\to p}+ \|G_2\|_{p\to p},\quad p\in[1,\infty).
\end{equation}
\end{lemma}

\begin{proof}
Equality \eqref{4.14} is obvious; the proof of \eqref{4.15} given in \cite{6} relies only on \eqref{2.2} and \eqref{2.21} and is therefore omitted.
\end{proof}

\begin{proof}[Proof of \thmref{thm3.7} and \corref{cor3.8}]
Both assertions were obtained in \cite{6} under conditions \eqref{1.2} and \eqref{2.5} (see Theorem 1.4 and Corollary 1.5 in \cite{6}). In \cite{6}, these conditions were required to guarantee the existence of a PFSS of equation \eqref{2.1} (according to \thmref{thm2.1.2}) and all remaining arguments in the proof of Theorem 1.4 and Corollary 1.5 in \cite{6} only rely on condition \eqref{1.2} and properties \eqref{2.2}, \eqref{2.3} and \eqref{2.4} of the PFSS equation \eqref{2.1}. In \thmref{thm3.7} and \corref{cor3.8}, the validity of \eqref{1.2} and the existence of a PFSS of equation \eqref{2.1} are a priori assumptions, and therefore the proofs of these assertions and Theorem 1.4 and Corollary 1.5 in \cite{6} coincide and are therefore omitted.
 \end{proof}

\begin{proof}[Proof of \lemref{lem3.9}]
For a given $x\in\mathbb R,$ consider the following function in $s\ge0$
 \begin{equation}\label{4.16}
F(s)=\int_{x-s}^{x+s}\frac{dt}{r(t)\rho(t)},\quad s\in[0,\infty]. \end{equation}
From \eqref{1.2}, \eqref{2.2}, \eqref{2.6}, \eqref{4.1} and \eqref{4.16}, it follows that
$$F(0)=0,\quad F(s)>0\qquad\text{for}\qquad s>0;\qquad F(\infty)=\infty.$$
In addition, $F(s)$ is monotone increasing and continuous for all $s\in[0,\infty).$ Hence, equation \eqref{3.7} has a unique solution $s(x),$ and it is clear that the function $s(x),$ $x\in\mathbb R$ is well-defined.
Let us check \eqref{3.8}. Let $t\in [0,s(x)],$ $x\in\mathbb R$ (the case $t\in[-s(x),0],$ $ x\in\mathbb R$ is treated in a similar way).

We have the following obvious equations:
\begin{gather*}
(x+t)-(t+s(x))\le x-s(x)\le x+s(x)\le (x+t)+(t+s(x)),\\
x-s(x)\le x+t-(s(x)-t)\le x+t+(s(x)-t)\le x+s(x).
\end{gather*}

Together with the definition of $s(x)$, $x\in\mathbb R,$ this implies
$$1=\int_{x-s(x)}^{x+s(x)}\frac{d\xi}{r(\xi)\rho(\xi)}\le
\int_{(x+t)-(t+s(x))}^{(x+t)+(t+s(x))}\quad \Rightarrow\quad s(x+t)\le s(x)+t=s(x)+|t|,$$
$$1=\int_{x-s(x)}^{x+s(x)}\frac{d\xi}{r(\xi)\rho(\xi)}\ge
\int_{(x+t)-(s(x)-t)}^{(x+t)+(s(x)-t)}\quad \Rightarrow\quad s(x+t)\ge s(x)-t=s(x)-|t|.$$

The obtained inequalities imply \eqref{3.8}. Furthermore, since the function $s(x),$ $x\in\mathbb R$ satisfies Lipschutz's condition \eqref{3.8}, it has a bounded derivative almost everywhere (see \cite[Ch.9,\S10]{1}).
Setting $s=s(x),$ $x\in\mathbb R$ in \eqref{3.7} and differentiating for almost all $x\in\mathbb R,$ we obtain
$$|s'(x)|=\frac{|r(x+s(x))\rho(x+s(x))-r(x-s(x))\rho(x-s(x))|}{r(x+s(x))\rho(x+s(x))
+r(x-s(x))\rho(x-s(x))}<1\ \Rightarrow\ \eqref{3.9}.$$

Let us now consider \eqref{3.10}. Since both equalities are checked in the same way,   we consider the second one. Assume to the contrary that $x-s(x)\not\rightarrow \infty$ as $x\to\infty.$ Then there exist a constant $c\in[1,\infty)$ and a sequence $\{x_n\}_{n=1}^\infty$ such that
$$x_n\ge c\quad\text{for}\quad n\ge 1,\quad x_n\to\infty\quad\text{as}\quad n\to\infty,
\quad\text{for }\quad x_n-s(x_n)\le c.$$

Then from \eqref{3.7} (for $s=s(x_n),$ $n\ge1)$ and \eqref{4.1}, it follows that
$$1=\int_{x_n-s(x_n)}^{x_n+s(x_n)}\frac{d\xi}{r(\xi)\rho(\xi)}\ge\int_c^{x_n}\frac{d\xi}{r(\xi)\rho(\xi)} \to\infty\quad\text{as}\quad n\to \infty.$$
This yields a contradiction, showing \eqref{3.10} and hence \eqref{3.11}. To prove \eqref{3.12}, define
$$g(x)=s(x)(1+|x|)^{-1},\quad x\in\mathbb R.$$
According to \eqref{3.11}, there is $x_0\gg1$ such that $g(x)\le 1$ for all $x\ge x_0.$ The function is continuous (see \eqref{3.8}) and positive on the finite interval $[-x_0,x_0]$. Hence it is bounded by a constant $c\ge1,$ so that $g(x)\le c$ for $x\in[-x_0,x_0],$ giving \eqref{3.12}.
\end{proof}

\begin{proof}[Proof of \lemref{lem3.11}]
Let $x\in\mathbb R,$ $\omega(x)=[x-s(x),x+s(x)],\quad \xi\in\omega(x),\quad t\in\omega(x).$
From \eqref{2.9}, it follows that
$$-\frac{1}{r(\xi)\rho(\xi)}\le\frac{\rho'(\xi)}{\rho(\xi)}\le\frac{1}{r(\xi)\rho(\xi)}
\quad\Rightarrow $$
\begin{align*}
-1&=-\int_{\omega(x)}\frac{d\xi}{r(\xi)\rho(\xi)}\le-\left|\int_x^t\frac{d\xi}{r(\xi)\rho(\xi)}\right| \le -\int_x^t\frac{d\xi}{r(\xi)\rho(\xi)}\le ln\frac{\rho(t)}{\rho(x)}\\
&\le\int_x^t\frac{d\xi}{r(\xi)\rho(\xi)}\le\left|\int_x^t\frac{d\xi}{r(\xi)\rho(\xi)}\right|\le\int_{ \omega(x)}\frac{d\xi}{r(\xi)\rho(\xi)}=1\quad\Rightarrow\quad \eqref{3.13}
\end{align*}
Now, for $t\in\omega(x),$ from \eqref{2.7}, \eqref{3.13} and the definition of $s(x),$ $x\in\mathbb R,$ it follows that
\begin{align*}
\frac{v(t)}{v(x)}&=\sqrt{\frac{\rho(t)}{\rho(x)}}\exp\left(\frac{1}{2}\int_x^t\frac{d\xi}{r(\xi)\rho(\xi)} \right)\le \sqrt e\exp\left(\frac{1}{2}\left|\int_x^t\frac{d\xi}{r(\xi)\rho(\xi)}\right|\right)\\
&\le\sqrt e\exp\left(\frac{1}{2}\int_{\omega(x)}\frac{d\xi}{r(\xi)\rho(\xi)}\right)=e,
\end{align*}
\begin{align*}
\frac{v(t)}{v(x)}&=\sqrt{\frac{\rho(t)}{\rho(x)}}
 \exp\left(\frac{1}{2}\int_x^t\frac{d\xi}{r(\xi)\rho(\xi)} \right)\ge\frac{1}{\sqrt e}\exp\left(-\frac{1}{2}\left|\int_x^t\frac{d\xi}{r(\xi)\rho(\xi)}\right|\right)\\
&\ge\frac{1}{\sqrt e}\exp\left(-\frac{1}{2}\int_{\omega(x)}\frac{d\xi}{r(\xi)\rho(\xi)}\right)=\frac{1}{e}.
\end{align*}

Inequalities \eqref{3.14} for $u(\cdot)$ are checked in a similar way.
\end{proof}

\renewcommand{\qedsymbol}{}

\begin{proof}[Proof of \thmref{thm3.12}]  \textit{Necessity}.
Since equation \eqref{1.1} is correctly solvable in $L_p,$ $p\in(1,\infty),$ we have $\|G\|_{p\to p} <\infty$ (see \thmref{thm3.7}). Below we successively use \eqref{4.15}, \eqref{2.38}, Lemmas~\ref{lem3.9} and \ref{lem3.11}, and \eqref{2.6}:
\begin{align*}
\infty&> 2\|G\|_{p\to p}\ge\|G_2\|_{p\to p}\ge\sup_{x\in\mathbb R}\left(\int_{-\infty}^x v(t)^pdt\right)^{1/p}\cdot\left(\int_x^\infty u(t)^{p'}dt\right)^{1/p'}\\
&\ge\sup_{x\in\mathbb R}\left[v(x)^p\int_{x-s(x)}^x\left(\frac{v(t)}{v(x)}\right)^pdt\right]^{1/p}\cdot \left[u(x)^{p'}\int_x^{x+s(x)}\left(\frac{u(t)}{u(x)}\right)^{p'}dt\right]^{1/p}\\
&\ge e^{-2}\sup_{x\in\mathbb R}\Big(\rho(x)\cdot s(x)^{\frac{1}{p}+\frac{1}{p'}}\Big)=e^{-2}\mathcal D\Rightarrow \mathcal D<\infty.
\end{align*}
\end{proof}

\renewcommand{\qedsymbol}{\openbox}

\begin{proof}[Proof of \thmref{thm3.12}]  \textit{Sufficiency}.
We need the following lemmas.

\begin{lemma}\label{lem4.5}
For any $x\in\mathbb R,$ the semi-axes $(-\infty,x]$ and $[x,\infty)$ admit $\mathbb R(x,s(\cdot))$-coverings (see Definition \ref{defn2.5.1}, \lemref{lem3.9} and \lemref{lem2.5.2}).
\end{lemma}

\begin{lemma}\label{lem4.6}
Let $x\in\mathbb R,$ and let $\{\Delta_n\}_{n=-\infty}^{-1}$ and $\{\Delta_n\}_{n=1}^\infty$ be
 $\mathbb R(x,s(\cdot))$-coverings of  the semi-axes $(-\infty,x]$ and $[x,\infty)$, respectively. Then we have the equalities
 \begin{equation}\label{4.17}
 \int_{\Delta_n^{(+)}}^{\Delta_{-1}^{(+)}}\frac{d\xi}{r(\xi)
 \rho(\xi)}=|n|-1\quad\text{for}\quad n\le -1\quad \int_{\Delta_1^{(-)}}^{\Delta_n^{(-)}}\frac{d\xi}{r(\xi)\rho(\xi)}=n-1\quad \text{for}\quad n\ge 1.
 \end{equation}
\end{lemma}

\begin{remark}\label{rem4.7} \lemref{lem4.5} immediately follows from \lemref{lem2.5.2}, \eqref{3.8} and \eqref{3.10} (see \cite{5}).

The proof of \lemref{lem4.6} is similar to the proof of equalities (3.32) in \cite{6} because the function $s(\cdot)$ is a `twin sister' of the function $d(\cdot)$ (see \lemref{lem2.3.5},  \eqref{2.30} and \lemref{lem3.9}).
\end{remark}

The proof of the sufficiency of the conditions of \thmref{thm3.12} is similar to the proof of Theorem 1.8 in \cite{6} because the latter proof only relies on the properties of the PFSS of equation \eqref{2.1}, Lemmas \ref{lem4.4}, \ref{lem4.5}, \ref{lem4.6}, \thmref{thm2.4.2}, and consists of checking the upper estimates in \eqref{3.16}.
\end{proof}

\begin{proof}[Proof of \corref{cor3.13}]
Below we show that each of the relations \eqref{3.17}, \eqref{3.18}, \eqref{3.19}, \eqref{3.20}, \eqref{3.21} implies the inequality $\mathcal D<\infty$ (see \eqref{3.15}), which guarantees the correct solvability of equation \eqref{1.1} in $L_p$, $p\in(1,\infty).$ So we proceed case by case.

1) Let $\sigma_1<\infty.$ Below, for $x\in\mathbb R,$ we use the definition of the function $s(\cdot)$ (see \eqref{3.7} and \eqref{3.13}):
$$1=\int_{x-s(x)}^{x+s(x)}\frac{dt}{r(t)\rho(t)}=\rho(x)\int_{x-s(x)}^{x+s(x)}\left(\frac{\rho(t)}{\rho(x)} \right)\frac{dt}{r(t)\rho^2(t)}\ge \frac{2}{e\sigma_1}\rho(x)s(x)\ \Rightarrow \ \mathcal D<\infty.$$

2) Let $\sigma_2<\infty.$ By \eqref{3.11}, there is $x_0\gg 1$ such that the following inequalities hold:
\begin{equation}\label{4.18}
\rho(x)s(x)\le\rho(x)|x|\le 2\sigma_2\quad\text{for}\quad |x|\ge x_0.
\end{equation}
Furthermore, the function $(\rho(x)s(x))$ is non-negative and continuous on the finite interval $[-x_0,x_0]$ (see \eqref{2.6} and \eqref{3.8}) and is therefore bounded there. Together with \eqref{4.18}, this gives the inequality $\mathcal D<\infty.$

3) Let $x\in\mathbb R. $ From \eqref{2.2}, \eqref{2.3} and \eqref{2.21}, it follows that
\begin{align}
\int_{-\infty}^{\infty}&q(t) G(x,t)dt =u(x)\int_{-\infty}^x q(t)v(t)dt+v(x)\int_x^\infty q(t)u(t)dt\label{4.19}\\
&=u(x)\int_{-\infty}^x(r(t)v'(t))'dt+v(x)\int_x^\infty (r(x)u'(t))'dt\nonumber\\
& =u(x)\left(r(t)v'(x)\bigg|_{-\infty}^x\right)
+v(x)\left((r(t)u'(t)\bigg|_x^\infty\right) \le r(x)(v'(x)u(x)-u'(x)v(x))=1.\nonumber\end{align}

Below, for $x\in\mathbb R$, we use relations \eqref{4.19}, \eqref{3.14} and the definition of the function $s(x),$ $x\in\mathbb R:$
\begin{align*}
1&\ge\int_{-\infty}^\infty q(t)G(x,t)dt\ge\int_{x-s(x)}^{x+s(x)}q(t)G(x,t)dt\\
&=u(x)v(x)\int_{x-s(x)}^xq(t)\left(\frac{v(t)}{v(x)}\right)dt+u(x)v(x)\int_x^{x+s(x)}q(t)
\left(\frac{u (t)}{u(x)}\right)dt\\
&\ge\frac{\rho(x)}{e}\int_{x-s(x)}^{x+s(x)}q(t)dt=\frac{1}{e}(\rho(x)\cdot s(x))\left[\frac{1}{2s(x)}\int_{x-s(x)}^{x+s(x)}q(t)dt\right]\\
&\ge\frac{2\sigma_3}{e}(\rho(x)s(x))\ \Rightarrow\ \mathcal D<\infty.
\end{align*}

4) The assertion follows from criterion \eqref{3.19}:
$$\sigma_3=\inf_{x\in\mathbb R}\left[\frac{1}{2s(x)}\int_{x-s(x)}^{x+s(x)}q(t)dt\right]\ge \sigma_4\inf_{x\in\mathbb R}\left[\frac{1}{2s(x)}\int_{x-s(x)}^{x+s(x)}1 dt\right]=\sigma_4>0 \Rightarrow \mathcal D< \infty.$$

5) From \eqref{3.21}, it follows that $\dfrac{1}{r}\in L_1.$

Furthermore, since \eqref{2.1} has a PFSS $\{u(x),v(x)\},$ $x\in\mathbb R,$ from \eqref{2.2} and \eqref{4.2} we obtain the inequalities:
\begin{gather*}
v(x)=u(x)\int_{-\infty}^x\frac{dt}{r(t)u^2(t)}\le\frac{1}{u(x)}\int_{-\infty}^x\frac{dt}{r(t)},\quad x\in\mathbb R,\\
u(x)=v(x)\int^{\infty}_x\frac{dt}{r(t)v^2(t)}\le\frac{1}{v(x)}\int^{\infty}_x\frac{dt}{r(t)},\quad x\in\mathbb R.
\end{gather*}

These inequalities imply the estimates:
\begin{gather*}
\rho(x)=u(x)\int_{-\infty}^x\frac{dt}{r(t)  }\le \left(\int_0^\infty\frac{dt}{r(t)}\right)^{-1}\cdot\int_{-\infty}^x\frac{dt}{r(t)}\cdot\int_x^\infty \frac{dt}{r(t)}\quad \text{for}\ x\le 0,\\
\rho(x)=u(x)v(x)\le\int_x^{\infty}\frac{dt}{r(t) }\le \left(\int^0_{-\infty}\frac{dt}{r(t)}\right)^{-1}\cdot\int_{-\infty}^x\frac{dt}{r(t)}\cdot\int_x^\infty \frac{dt}{r(t)}\quad \text{for}\ x\ge 0.
\end{gather*}

Hence,
\begin{equation}\label{4.20}
\rho(x)\le\tau\int_{-\infty}^x\frac{dt}{r(t)}\cdot\int_x^\infty\frac{dt}{r(t)},\quad x\in\mathbb R,
\end{equation}
where $$ \tau=\max\left\{\left(\int_{-\infty}^0\frac{dt}{r(t)}\right)^{-1},\ \left(\int_0^\infty\frac{dt}{r(t)}\right)^{-1}\right\}.$$
The assertion now follows from \eqref{4.20} and criterion \eqref{3.17}
\end{proof}

\begin{proof}[Proof of \corref{cor3.14}]
By 2) (see \eqref{3.22}), for any $n\ge 1,$ there is $x_0(u)\gg1$ such that $r(x)\rho(x)\ge n$ for $|x|\ge x_0(n).$ Furthermore, according to \eqref{3.10}, there is $x_1(n)\gg1$ such that
 $$[x-s(x),x+s(x)]\cap[-x_0(n),x_0(n)]=\emptyset\quad\text{for}\quad |x|\ge x_1(n).$$
 Set $x_2(n)=\max\{x_0(n),x_1(n)\}.$ Then for $|x|\ge x_2(n),$ from 2) and \lemref{lem3.9}, it follows that
$$1=\int_{x-s(x)}^{x+s(x)}\frac{dt}{r(t)\rho(t)}\le\frac{2}{n}s(x)\ \Rightarrow\ s(x)\ge\frac{n}{2}\quad\text{for}\quad |x|\ge x_2(n),$$
i.e., $s(x)\to\infty$ as $|x|\to\infty.$ Besides, from Lemmas \ref{lem3.9} and \ref{lem3.11}, we obtain
\begin{gather*}
1=\int_{x-s(x)}^{x+s(x)}\frac{dt}{r(t)\rho(t)}\ge\frac{1}{e}\, \frac{1}{\rho(x)}\int_{x-s(x)}^{x+s(x)}\frac{dt}{r(t)}\ge\frac{2}{er_0}\,\frac{s(x)}{\rho(x)}\quad\text{(see 1) in \eqref{3.22})}\ \Rightarrow\\
\rho(x)\ge c^{-1}s(x)\quad\text{for}\quad |x|\ge x_2(n),\quad c\in [1,\infty)\ \Rightarrow\\
\rho(x)s(x)\ge c^{-1} s^2(x)\to\infty\quad\text{as}\quad |x|\to\infty\ \Rightarrow \ \mathcal D=\infty.
\end{gather*}
It remains to refer to \thmref{thm3.12}.
\end{proof}

\begin{proof}[Proof of \thmref{thm3.17}]
Below, we assume that all conditions in \thmref{thm3.17} are satisfied and do not quote them. To prove the theorem, we establish the inequalities
\begin{equation}\label{4.21}
c^{-1}\mathcal D_1\le\mathcal D \le c(p)\mathcal D_1,\quad p\in(1,\infty);\quad c,c(p)\in[1,\infty).
\end{equation}
Here
\begin{equation}\label{4.22}
 \mathcal D=\sup_{x\in\mathbb R}(\rho(x)s(x)),\quad\mathcal D_1=\sup_{x\in\mathbb R}(\rho_1(x)s_1(x)),
\end{equation}
$\rho(\cdot),$ $\rho_1(\cdot)$ are the functions generating the PFSS of equations \eqref{2.1} and \eqref{2.11}, respectively (under conditions \eqref{1.2} and \eqref{3.25}); $s(x) $ and $s_1(x)$ are the solutions in $s\ge0$ (for a fixed $x\in\mathbb R) $ of the equations
\begin{equation}\label{4.23}
\int_{x-s}^{x+s}\frac{dt}{r(t)\rho(t)}=1,\qquad \int_{x-s}^{x+s}\frac{dt}{r(t)\rho_1(t)}=1,
\end{equation}
respectively (see \lemref{lem3.9}).

Clearly, \thmref{thm3.17} follows from \eqref{4.21} and \thmref{thm3.12}. So let us go to \eqref{4.21}. We need the inequalities (see \eqref{2.23}):
\begin{align}
\|G_2\|_{p\to p}&\le c(p)\sup_{x\in\mathbb R}\left[\int_{-\infty}^x \rho(t)\exp\left(-(p-1)\int_t^x\frac{d\xi}{ r(\xi)\rho(\xi)}\right)dt\right]^{1/p}\label{4.24}\\
&\quad \cdot\left[\int_x^\infty\rho(t)\exp\left(-\int_x^t\frac{d\xi}{r(\xi)\rho(\xi)}\right)dt\right]^{1/p'}\quad\text{for}
\quad p\in(1,2],\nonumber
\end{align}
\begin{align}
\|G_2\|_{p\to p}&\le c(p)\sup_{x\in\mathbb R}\left[\int_{-\infty}^x \rho(t)\exp\left(-\int_t^x\frac{d\xi}{ r(\xi)\rho(\xi)}\right)dt\right]^{1/p}\label{4.25}\\
&\quad \cdot\left[\int_x^\infty\rho(t)\exp\left(-(p'-1)\int_x^t\frac{d\xi}{r(\xi)\rho(\xi)}\right)dt\right]^{1/p'}\quad\text{for}
\quad p\in(2,\infty),\nonumber
\end{align}
\begin{equation}\label{4.26}
\|G_2\|_{p\to p} \ge  \sup_{x\in\mathbb R}\left[\int_{-\infty}^x v(t)^pdt\right]^{1/p}\cdot\left[\int_x^ \infty u(t)^{p'}dt\right]^{1/p'}\ \quad\text{for}\quad p\in(1,\infty)
\end{equation}
(see relations (3.30), (3.34) and (3.35) in \cite{6}).
We can use these inequalities in our proof because they were obtained in \cite{6} from the properties of the PFSS of equation \eqref{2.1} (see Definition \ref{defn2.1.1}) and Theorems \ref{thm2.4.2} and \ref{thm2.1.3}.

To prove the upper estimate in \eqref{4.24}, denote by $\{\Delta\}_{n=-\infty}^{-1}$ and $\{\Delta_n\}_{n=1}^\infty$ the segments of $\mathbb R(x,s(x)$-coverings of the semi-axes $(-\infty,x]$ and $[x,\infty)$, respectively (see Definition \ref{defn2.5.1}, \lemref{lem2.5.2} and \eqref{3.10}), constructed for the function $\rho_1(\cdot);$ by the letter ``$a$'' $(a\in[1,\infty))$, for the sake of clarity, we denote below the constant for which, by the conditions of the theorem, the following inequalities hold:
\begin{equation}\label{4.27}
a^{-1}\rho(\xi)\le\rho_1(\xi)\le ap(\xi)\quad \xi\in\mathbb R.
\end{equation}

Below, for $p\in(1,2),$ we successively use relations   \eqref{3.15}, \eqref{3.16}, \eqref{3.13}, \eqref{4.24} and \eqref{4.17}:
\begin{align*}
c^{-1}\mathcal D&\le\|G_2\|_{p\to p}\le c(p)\sup_{x\in\mathbb R}\left[\int_{-\infty}^x\rho(t)\exp\left(-(p-1)\int_x^t\frac{d\xi}{r(\xi)\rho(\xi)}\right)dt\right]^{1/p}\\
&\quad\cdot\left[\int_x^\infty\rho(t)\exp\left(-\int_x^t\frac{d\xi}{r(\xi)\rho(\xi)}\right)dt\right]^{1/p'}\\
&\le c(p)\sup_{x\in\mathbb R}\left[a\int_{-\infty}^x\rho_1(t)\exp\left(-\frac{p-1}{a}\int_t^x\frac{d\xi}{r(\xi)\rho_1(\xi)}\right)dt\right ]^{1/p}\\
&\quad\cdot\left[a\int_x^\infty\rho_1(t)\exp\left(-\frac{1}{a}\int_x^t\frac{d\xi}{r(\xi)\rho_1(\xi)}\right)dt\right] ^{1/p'}\\
&\le ac(p)\sup_{x\in\mathbb R}\left[\sum_{n=-\infty}^{-1}\int_{\Delta_n}\rho_1(t)\exp\left(-\frac{p-1}{a}\int_t^{\Delta_{-1}^{(+)}}\frac{d\xi} {r(\xi)\rho_1(\xi)}\right)dt\right]^{1/p}\\
&\quad\cdot\left[\sum_{n=1}^\infty\int_{\Delta_n}\rho_1(t)\exp\left(-\frac{1}{a}\int_{\Delta_1^{(-)}}^t\frac{ d\xi}{r(\xi)\rho_1(\xi)}\right)dt\right]^{1/p'}\\
&\le ac(p)\sup_{x\in\mathbb R}\left[\sum_{n=-\infty}^{-1}2e\rho_1(x_n)s_1(x_n)\exp\left(-\frac{p-1}{a}\int_ {\Delta_n^{(+)}}^{\Delta_{-1}^{(+)}} \frac{d\xi}{r(\xi)\rho_1(\xi)}\right)\right]^{1/p}\\
&\quad\cdot\left[\sum_{n=1}^\infty 2e\rho_1(x_n)s_1(x_n)\exp\left(-\frac{1}{a}\int_{\Delta_1^{(-)}}^{\Delta_n ^{(-)}}\frac{d\xi}{r(\xi)\rho_1(\xi)}\right)\right]^{1/p'}\\
&\le 2eac(p)\left[\mathcal D_1\sum_{n=1}^\infty\exp\left(-\frac{p-1}{a}(n-1)\right)\right]^{1/p}\cdot\left[ \mathcal D_1\sum_{n=1}^\infty\exp\left(-\frac{n-1}{2}\right)\right]^{1/p'}=c(p)\mathcal D_1,
\end{align*}
i.e., $\mathcal D\le c(p)\mathcal D_1,$ as required. The upper estimate in \eqref{4.21} for $p\in(2,\infty)$ is proved in a similar way, with the help of \eqref{4.25} instead of \eqref{4.24}.

Let us go to the lower estimate in \eqref{4.21}.
Below, for $p\in(1,\infty),$ we successively use \eqref{3.16}, \thmref{thm2.4.2}, \eqref{4.26} and Lemmas \ref{lem3.9} and \ref{lem3.11}:
\begin{align*}
&c(p)\mathcal D\ge\|G_2\|_{p\to p}\ge\sup_{x\in\mathbb R}\left[\int_{-\infty}^x v(t)^pdt\right]^{1/p}\cdot \left[\int_x^\infty u(t)^{p'}dt\right]^{1/p'}\\
&=\sup_{x\in\mathbb R}\left[\int_{-\infty}^x \rho(t)^{p/2}\exp\left(\frac{p}{2}\int_{x_0}^t\frac{d\xi}{r(\xi)\rho(\xi)}\right)dt\right]^{1/p}\cdot\left[\int_ x^\infty \rho(t)^{p'/2}\exp\left(-\frac{p'}{2}\int_{x_0}^t \frac{d\xi}{r(\xi)\rho(\xi)}\right)dt\right]^{1/p'}\\
&=\sup_{x\in\mathbb R}\left[\int_{-\infty}^x\rho(t)^{p/2}\exp\left(-\frac{p}{2}\int_t^x\frac{d\xi}{r(\xi)\rho(\xi)}\right)dt\right] ^{1p}\cdot \left[\int_x^\infty\rho(t)^{p'/2}\exp\left(-\frac{p'}{2}\int_x^t\frac{d\xi}{r(\xi)\rho(\xi)}\right)dt\right]^{1/p'}\\
&\ge\sup_{x\in\mathbb R}\left[\int_{x-s_1(x)}^x\rho(t)^{p/2}\exp\left(-\frac{p}{2}\int_t^x\frac{d\xi}{r(\xi)\rho(x)}\right)\right]^{1/p}\\
&\quad\cdot\left[\int_x^{x+s_1(x)}\rho(t)^{p'/2}\exp\left(-\frac{p'}{2}\int_x^t\frac{d\xi}{r(\xi)\rho(\xi)}\right) dt\right]^{1/p'}\\
&\ge\sup_{x\in\mathbb R}\left[a^{-p/2} e^{-p/2}\rho_1(x)^{p/2}\exp\left(-\frac{ap}{2}\int_{x-s_1(x)}^x\frac{d\xi}{r(\xi)\rho_1(\xi)}\right)s_1(x)\right] ^{1/p}\\
&\quad\cdot\left[a^{-p'/2}e^{-p'/2}\rho_1(x)^{p'/2}\exp\left(-\frac{ap'}{2}\int_x^{x+s_1(x)}\frac{d\xi}{r(\xi)\rho_1(\xi)}\right) s_1(x)\right]^{1/p'}\\
&=(ae)^{-1}\exp\left(-\frac{a}{2}\right)\sup_{x\in\mathbb R}(\rho_1(x)s_1(x))=c^{-1}\mathcal D_1\ \Rightarrow\ \eqref{4.21}.
\end{align*}
\end{proof}

\begin{proof}[Proof of \thmref{thm3.18}]
The relations $\rho(x)\asymp\rho_1(x)$ for $x\in(-\infty,0]$ and $[0,\infty)$ are proved in the same way. They imply the assertion of the theorem. Therefore, below we only consider the case $[x,\infty).$ Since for equations \eqref{2.1} and \eqref{2.11} the Hartman-Wintner problem for $x\to\infty$ is solvable, there exists a FSS $\{\hat u(x),\hat v(x)\},$ $x\in\mathbb R, $ of equation \eqref{2.1} for which equalities \eqref{2.12}--\eqref{2.14} hold as $x\to\infty.$ In addition, since $\{u(x),v(x)\},$ $x\in\mathbb R,$ is a PFSS of equation \eqref{2.1}, there are constants $\alpha$ and $\beta$ $(0<|\alpha|+|\beta|<\infty)$ such that
\begin{equation}\label{4.28}
\hat u(x)=\alpha u(x)+\beta v(x),\quad x\in\mathbb R.
\end{equation}
Assume that in \eqref{4.28} we have $\beta\ne0.$ By \eqref{2.12}, there is $x_1\gg1$ such that
\begin{equation}\label{4.29}
2^{-1} u_1(x)\le\hat u(x)\le 2u_1(x)\quad\text{for}\quad x\ge x_1.
\end{equation}
Then from \eqref{4.4}, \eqref{4.28} and \eqref{4.29}, it follows that
\begin{equation}\label{4.30}
\int_{x_1}^\infty\frac{dt}{r(t)\hat u(x)^2}=\int_{x_1}^\infty\frac{1}{r(t)u_1^2(t)}\left(\frac{u_1(t)}{\hat u(t)}\right)^2dt\ge\frac{1}{4}\int_{x_1}^\infty\frac{dt}{r(t)u_1^2(t)}=\infty.
\end{equation}

Furthermore, according to \eqref{2.4}, by enlarging $x_1$ if needed, we get
\begin{equation}\label{4.31}
\left|\frac{\alpha}{\beta}\right|\frac{u(x)}{v(x)}\le\frac{1}{2}\quad\text{for}\quad x\ge x_1.
\end{equation}
Now, from \eqref{4.28}, \eqref{4.30}, \eqref{4.31} and \eqref{4.4}, we obtain
\begin{align*}
\infty&=
\int_{x_1}^\infty\frac{dt}{r(t)\hat u^2(t)}=\int_{x_1}^\infty
\frac{dt}{r(t)(\alpha u(t)+\beta v(t))^2}\le\int_{x_1}^\infty\frac{dt}{r(t)v^2(t)\left(1-\left|\frac{\alpha}{\beta}\right|\frac{u(t)}{v(t)}\right)^2}\\
&\le\frac{4}{\beta^2}\int_{x_1}^\infty\frac{dt}{r(t)v^2(t)}<\infty.
\end{align*}
We get a contradiction. Hence, $\beta=0$ and therefore $\alpha>0$ (see \eqref{4.28}, \eqref{4.29}, \eqref{2.12}, and \eqref{2.2}). Therefore, from \eqref{4.28} (recall that $\beta=0$) and \eqref{2.12}, we obtain
\begin{equation}\label{4.32}
1=\lim_{x\to\infty}\frac{\hat u(x)}{u_1(x)}=\alpha\lim_{x\to\infty}\frac{u(x)}{u_1(x)}\ \Rightarrow \ \lim_{x\to\infty}\frac{u(x)}{u_1(x)} =\frac{1}{\alpha},\quad \alpha\in(0,\infty).
\end{equation}

Note that since $\{u(x),v(x)\},$ $x\in\mathbb R$, and $\{u_1(x),v_1(x)\},$ $x\in\mathbb R,$ are PFSS of equations \eqref{2.1} and \eqref{2.11}, respectively, we obtain from Definition \ref{defn2.1.1}:
\begin{equation}\label{4.33}
\left(\frac{v(x)}{u(x)}\right)'=\frac{1}{r(x)u^2(x)},\quad x\in\mathbb R;\quad\left(\frac{v_1(x)}{u_1(x)}\right)'=\frac{1}{r(x)u_1^2(x)},\quad x\in\mathbb R.
\end{equation}
This implies that
\begin{equation}\label{4.34}
v(x)=\frac{v(o)}{u(o)} u(x)+u(x)\int_0^x\frac{dt}{r(t)u^2(t)},\quad x\ge0,
\end{equation}
\begin{equation}\label{4.35}
v_1(x)=\frac{v_1(o)}{u_1(o)}u_1(x)+ u_1(x) \int_0^x\frac{dt}{r(t)u_1^2(t)},\quad x\ge0.
\end{equation}

Below we use \eqref{4.32}, \eqref{4.34}, \eqref{4.35} and L'H\^{o}pital's rule (see \eqref{4.4}):
\begin{equation}\label{4.36}
\lim_{x\to\infty}\frac{v(x)}{v_1(x)}=\lim_{x\to\infty}\frac{u(x)}{u_1(x)}\frac{\frac{v(o)}{u(o)}+ \int_0^x\frac{dt}{r(t)u^2(t)}}{\frac{v_1(o)}{u_1(o)}+\int_0^x\frac{dt}{r(t)u^2(t)}}=\alpha^{-1}\lim_{x\to\infty} \frac{r(x)u_1^2(x)}{r(x)u^2(x)}=\alpha.
\end{equation}
By \eqref{4.32} and \eqref{4.36}, we thus get
\begin{equation}\label{4.37}
\lim_{x\to\infty}\frac{\rho(x)}{\rho_1(x)}=\lim_{x\to\infty}\frac{u(x)}{u_1(x)}\cdot\frac{v(x)}{v_1(x)}=1.
\end{equation}

In particular, according to \eqref{4.37}, there is $x_2\gg1$ such that
\begin{equation}\label{4.38}
2^{-1}\rho_1(x)\le\rho(x)\le 2\rho(x)\quad\text{for}\quad x\ge x_1.
\end{equation}
Since the function
$$f(x)=\frac{\rho(x)}{\rho_1(x)},\quad x\in[0,x_2]$$
is continuous and positive for $x\in[0,x_2],$ there is $c\ge2$ such that
\begin{equation}\label{4.39}
c^{-1}\le f(x)\le c\quad\text{for}\quad x\in [0,x_1],
\end{equation}
and by \eqref{4.38} and \eqref{4.39}, we obtain the relation $\rho(x)\asymp \rho_1(x),$ $x\in[0,\infty),$ as required
\end{proof}

\begin{proof}[Proof of \thmref{thm3.19}]
First note that equation \eqref{2.1} has a PFSS $\{u(x),v(x)\},$ $x\in\mathbb R$ due to \thmref{thm2.1.2}. Furthermore, since $\frac{1}{r}\in L_1,$ equation \eqref{3.2} has a PFSS $\{u_1(x),v_1(x)\}$, $x\in\mathbb R,$ of the form \eqref{3.1}, and we have the inequalities (see \eqref{3.1}):
\begin{equation}\label{4.40}
\rho_1(x)=u_1(x)v_1(x)\le\frac{1}{w_0}\int_{-\infty}^x\frac{dt}{r(t)}\cdot\int_x^\infty \frac{dt}{r(t)}\le\begin{cases}\int_{-\infty}^x\frac{dt}{r(t)},\quad & x\le0\\ \int_x^\infty\frac{dt}{r(t)},\quad &x\ge0.\end{cases}
\end{equation}
Then, from \eqref{3.26}, \eqref{4.40} and \thmref{thm2.2.3}, it follows that the Hartman-Wintner problems for equations \eqref{2.1} and \eqref{3.2}, for $x\to-\infty$ and $x\to\infty,$ are solvable, and therefore the functions $\rho(x)$ and $\rho_1(x),$ $x\in\mathbb R,$ generating the PFSS of equations \eqref{2.1} and \eqref{3.2}, are weakly equivalent. It remains to refer to \thmref{thm3.17}.
\end{proof}

\section{Example}

First recall that our goal is the study of problem I)-II), $p\in(1,\infty)$ within the framework of \eqref{1.4}. Therefore, although some of our statements are relevant also in the case \eqref{1.3} along with \eqref{1.4} (see, e.g., \thmref{thm3.12}), we restrict our attention to the case \eqref{1.4} and do not give any example of the stud of the question on I)-II) under the condition \eqref{1.3} ( the latter case is considered in \cite{3,4,6,7,8,9}).

Below we consider the equation
\begin{equation}\label{5.1}
-(r(x)y'(x))'+q(x)y(x)=f(x),\quad x\in\mathbb R.
\end{equation}
Here and in the sequel, $f(\cdot)\in L_p,$ $p\in(1,\infty)$ and
\begin{equation}\label{5.2}
r(x)=(1+x^2)^\alpha,\quad\alpha>\frac{1}{2};\quad q(x)=\frac{1}{(1+x^2)^\beta},\quad \beta>\frac{1}{2},\quad x\in\mathbb R.
\end{equation}
The question on I)-II) $p\in(1,\infty)$ is denoted below as problem \eqref{5.1}-\eqref{5.2}. Our study of \eqref{5.1}-\eqref{5.2} is based on Theorems \ref{thm3.12}, \ref{thm3.19} and \corref{cor3.13}. For the reader's convenience, we divide the exposition into stages 1), 2) and 3) and comment on each stage separately.

1)\ \textit{Homogeneous equations}

Below, we use the fact that by \eqref{5.2}, each of the equations
\begin{equation}\label{5.3}
(r(x)z'(x))'=q(x)y(x),\quad x\in\mathbb R,
\end{equation}
\begin{equation}\label{5.4}
(r(x)z'(x))'=0,\quad x\in\mathbb R
\end{equation}
has a PFSS (see Theorems \ref{thm2.1.2} and \ref{thm3.1}). Below we do not refer to this fact.

2)\ \textit{Solving problem \eqref{5.1}-\eqref{5.2} for $\alpha\ge1$}

If $\frac{1}{r}\in L_1,$ one can apply criterion \eqref{3.21} for the correct solvability of problem I)-II), which is convenient because this allows one to avoid the proofs of estimates for the functions $\rho(x),$ $x\in\mathbb R$, and $s(x),$ $x\in\mathbb R$ (cf. \thmref{thm3.12}). In our case, the functions in \eqref{5.2} are even, and therefore we have the following relations (see \eqref{3.21}):
$$
\sigma_5=\sup_{x\in\mathbb R}\left[r(x)\left(\int_{-\infty}^x\frac{dt}{r(t)}\right)^2\left(\int_x ^\infty\frac{dt}{r(t)}\right)^2\right]\le c\sup_{x\ge0}\left[r(x)\left(\int_x^\infty\frac{dt}{r(t)}\right)^2\right].
$$

Denote (see \eqref{5.2}):
\begin{equation}\label{5.5a}
F(x)=(1+x^2)^\alpha\left(\int_x^\infty\frac{dt}{(1+t^2)^\alpha}\right)^2,\quad x\ge0.
\end{equation}
L'H\^opital's rule implies that
\begin{equation}\label{5.6}
\lim_{x\to\infty}\sqrt{F(x)}=\begin{cases} 1,&\quad\text{if}\quad \alpha=1\\
0,&\quad \text{if}\quad \alpha>1.\end{cases}
\end{equation}

By \eqref{5.6}, there is $x_0\gg1$ such that
\begin{equation}\label{5.7}
0<F(x)\le 2\quad\text{for}\quad x\ge x_0.
\end{equation}
Since the continuous positive function $F(x)$ is bounded on $[0,x_0],$ by \eqref{5.7} $F(\cdot)$ is absolutely bounded on the semi-axis $[0,\infty).$ This implies (see \eqref{5.5a}) that $\sigma_5<\infty,$ i.e., for $\alpha\ge1$ problem (5.1)-(5.2) is correctly solvable in $L_p,$ $p\in(1,\infty).$

3)\ \textit{Problem I)-II), $p\in(1,\infty)$ for the model equation \eqref{3.17}}

Thus, in the case $\alpha\in\left(\frac{1}{2},1\right),$ \corref{cor3.13} does not give an answer to the question on I)-II), $p\in(1,\infty)$. Therefore, to complete our study of \eqref{5.1}-\eqref{5.2}, we have to first study problem I)-II), $p\in(1,\infty),$ for the
model equation \eqref{3.27} and then, applying \thmref{thm3.19}, get information on problem \eqref{5.1}-\eqref{5.2}. By \thmref{thm3.1}, the function $\rho(x),$ $x\in\mathbb R,$ generating the PFSS of equation \eqref{3.1}, is of the form
\begin{equation}\label{5.8}
\rho(x)=\frac{1}{w_0}\int_{-\infty}^x\frac{dt}{(1+t^2)^\alpha}\cdot\int_x^\infty\frac{dt}{(1+t^2)^ \alpha},\quad x\in\mathbb R,\quad w_0=\int_{-\infty}^\infty \frac{dt}{(1+t^2)^\alpha}.
\end{equation}

L'H\^opital's rule implies that
\begin{equation}\label{5.9}
\lim_{|x|\to\infty}\rho(x)|x|^{2\alpha-1}=\frac{1}{2\alpha-1},\quad \alpha\in\left(\frac{1}{2},1\right)
\end{equation}
which gives an asymptotic formula for $|x|\to\infty:$
\begin{equation}\label{5.10}
r(x)\rho(x)=\frac{1+\delta(x)}{2\alpha-1}|x|,\quad \lim_{|x|\to\infty}\delta(x)=0.
\end{equation}

To estimate $s(x)$ for $x\gg1$ (see \lemref{lem3.9}), we use equation \eqref{3.7} and relations \eqref{3.10} and \eqref{5.10}:
\begin{align}
&1=\int_{x-s(x)}^{x+s(x)}\frac{dt}{r(t)s(t)}=\int_{x-s(x)}^{x+s(x)}\frac{2\alpha-1}{1+\delta(t)}\, \frac{dt}{t}\le 2(2\alpha-1)\int_{x-s(x)}^{x+s(x)}\frac{dt}{t}\ \Rightarrow\nonumber\\
&c(\alpha):=\frac{1}{2(2\alpha-1)}\le\ln\frac{x+s(x)}{x-s(x)},\quad x\gg1\ \Rightarrow\ s(x)\ge \frac{e^{c(\alpha)}-1}{e^{c(\alpha)}+1}x,\quad x\gg1.\label{5.11}
\end{align}
Thus, for the model equation \eqref{3.27}, problem I)-II), $p\in(1,\infty)$ for $\alpha\in\left(\frac{1}{2},1\right)$ is not correctly solvable in light of \thmref{thm3.12} because (see \eqref{3.15}, \eqref{5.10} and \eqref{5.11})
\begin{align*}
\mathcal D&=\sup_{x\in\mathbb R}(\rho(x)s(x)\ge\sup_{x\gg1}(\rho(x)s(x))\ge\frac{e^{c(\alpha)}-1} {e^{c(\alpha)+1}}\sup_{x\gg1}\frac{1+\delta(x)}{2\alpha-1}\, \frac{x^2}{(1+x^2)}\alpha\\
&\ge\frac{1}{c}\sup_{x\gg1}x^{2(1-\alpha)}=\infty.
\end{align*}

Together with \thmref{thm3.19}, this implies that problem \eqref{5.1}-\eqref{5.2}, $p\in(1,\infty),$ for $\alpha\in\left(\frac{1}{2},1\right)$ is also not correctly solvable.

\end{document}